\documentclass[xcolor=x11names,reqno,12pt]{amsart}
 \usepackage{amsmath}
 \usepackage{amsthm}
 \usepackage{amssymb}
 \usepackage{latexsym,longtable}
 \usepackage{graphicx}
 \usepackage{multicol}
 \usepackage{mathrsfs}
 \usepackage{young}
 \usepackage[vcentermath,enableskew,stdtext]{youngtab}
 \usepackage[left=1.1in,right=1.1in,top=1.08in,bottom=1.1in]{geometry}
 \usepackage[all]{xy}
    \SelectTips{cm}{10}     
    \everyxy={<2.5em,0em>:} 
 \usepackage{fancyhdr}      
 \usepackage{multicol}
 \usepackage{multirow}
 \usepackage{enumitem}
 \usepackage{bm}
 \usepackage{stmaryrd}
 \usepackage{etex}
 \usepackage{tikz}
 \usepackage{float}
 \usepackage{array}
 \usepackage{colortbl}
 \usepackage{mathtools}
 \restylefloat{figure}
\numberwithin{equation}{section}
 \usetikzlibrary{snakes}
\usepackage{url}
 \usepackage{ytableau}

\usepackage{amsfonts,epsfig,color}
\makeatletter
\newcommand*{\centerfloat}{%
  \parindent \z@
  \leftskip \z@ \@plus 1fil \@minus \textwidth
  \rightskip\leftskip
  \parfillskip \z@skip}
\makeatother

\theoremstyle{plain}
\newtheorem{theorem}{Theorem}[section]
\newtheorem{corollary}[theorem]{Corollary}

\newtheorem*{conjecture*}{Conjecture}

\theoremstyle{definition}

\newtheorem{remark}[theorem]{Remark}

\newcommand{\ignore}[1]{}

\newcommand{\be}{\begin{equation}}
\newcommand{\ee}{\end{equation}}

\newcommand{\crc}[1]{#1\star}

\newlength{\mycellsize}
\newcommand\mytbl[1]{
\vcenter{
\let\\=\cr
\baselineskip=-16000pt \lineskiplimit=16000pt \lineskip=0pt
\halign{&\mytblcell{##}\cr#1\crcr}}}

\newcommand{\mytblcell}[1]{{%
\def \arg{#1}\def \void{}%
\ifx \void \arg
\vbox to \mycellsize{\vfil \hrule width \mycellsize height 0pt}%
\else \unitlength=\mycellsize
\begin{picture}(1,1)
\put(0,0){\makebox(1,1){$#1\vphantom{\crc{#1}}$}}
\put(0,0){\line(1,0){1}}
\put(0,1){\line(1,0){1}}
\put(0,0){\line(0,1){1}}
\put(1,0){\line(0,1){1}}
\end{picture}%
\fi}}

\newlength{\cellsize}
\newcommand\mytableau[1]{
\vcenter{
\let\\=\cr
\baselineskip=-16000pt \lineskiplimit=16000pt \lineskip=0pt
\halign{&\mytableaucell{##}\cr#1\crcr}}}

\newcommand{\mytableaucell}[1]{{%
\def \arg{#1}\def \void{}%
\ifx \void \arg
\vbox to \cellsize{\vfil \hrule width \cellsize height 0pt}%
\else \unitlength=\cellsize
\begin{picture}(1,1)
\put(0,0){\makebox(1,1){$#1\vphantom{\crc{#1}}$}}
\put(0,0){\line(1,0){1}}
\put(0,1){\line(1,0){1}}
\put(0,0){\line(0,1){1}}
\put(1,0){\line(0,1){1}}
\end{picture}%
\fi}}

\newcommand\boldtableau[1]{
\vcenter{
\let\\=\cr
\baselineskip=-16000pt \lineskiplimit=16000pt \lineskip=0pt
\halign{&\boldtableaucell{##}\cr#1\crcr}}}

\newcommand{\boldtableaucell}[1]{{%
\def \arg{#1}\def \void{}%
\ifx \void \arg
\vbox to \cellsize{\vfil \hrule width \cellsize height 0pt}%
\else \unitlength=\cellsize
\begin{picture}(1,1)
\put(0,0){\makebox(1,1){$\mathbf{#1\vphantom{\crc{#1}}}$}}
\put(0,0){\line(1,0){1}}
\put(0,1){\line(1,0){1}}
\put(0,0){\line(0,1){1}}
\put(1,0){\line(0,1){1}}
\end{picture}%
\fi}}

\title{A Note on the Higher order Tur\'{a}n inequalities for $k$-regular partitions}

\keywords{}

\begin{document}

\author{William Criag}
\address{Department of Math, University of Virginia, Charlottesville, VA 22904}
\email{wlc3vf@virginia.edu}

\author{Anna Pun}
\address{Department of Math, University of Virginia, Charlottesville, VA 22904}
\email{annapunying@gmail.com}

\thanks{We acknowledge the support of NSF Grant \# DMS-1601306.}

\begin{abstract}
Nicolas \cite{Nicolas78} and DeSalvo and Pak \cite{DP15} proved that the partition function $p(n)$ is log concave for $n \geq 25$. Chen, Jia and Wang \cite{CJW19} proved that $p(n)$ satisfies the third order Tur\'{a}n inequality, and that the associated degree 3 Jensen polynomials are hyperbolic for $n \geq 94$. Recently, Griffin, Ono, Rolen and Zagier \cite{GORZ_1} proved more generally that for all $d$, the degree $d$ Jensen polynomials associated to $p(n)$ are hyperbolic for sufficiently large $n$. In this paper, we prove that the same result holds for the $k$-regular partition function $p_k(n)$ for $k \geq 2$. In particular, for any positive integers $d$ and $k$, the order $d$ Tur\'{a}n inequalities hold for $p_k(n)$ for sufficiently large $n$. The case when $d = k = 2$ proves a conjecture by Neil Sloane that $p_2(n)$ is log concave.
\end{abstract}

\maketitle

\section{Introduction and Statement of results}

The Tur\'{a}n inequality (or sometimes called Newton inequality) arises in the study of real entire functions in Laguerre-P\'{o}lya class, which is closely related to the study of the Riemann hypothesis \cite{Dimitrov98, Szego48}. It is well known that the Riemann hypothesis is true if and only if the Riemann Xi function is in the Laguerre-P\'{o}lya class, where the Riemann Xi function is the entire order $1$ function defined by $$\Xi(z) := \frac{1}{2}\bigg(-z^2 - \frac{1}{4}\bigg)\pi^{\frac{iz}{2}-\frac{1}{4}}\Gamma\bigg(-\frac{iz}{2}+\frac{1}{4}\bigg)\zeta\bigg(-iz+\frac{1}{2}\bigg).$$

\noindent A necessary condition for the Riemann Xi function to be in the Laguerre-P\'{o}lya class is that the Maclaurin coefficients of the Xi function satisfy both the Tur\'{a}n and higher order Tur\'{a}n inequalities \cite{Dimitrov98, PS1914}.

A sequence $\{a_m\}_{m = 0}^{\infty}$ is \textit{log concave} if it satisfies the (second order) Tur\'{a}n inequality $a_m^2 \geq a_{m-1}a_{m+1}$ for all $m \geq 1$. The sequence $\{ a_m \}_{m = 0}^{\infty}$ satisfies the third order Tur\'{a}n  inequality if for $m \geq 1$, we have $$4(a_m^2 -a_{m-1}a_{m+1})(a_{m+1}^2-a_{m}a_{m+2}) \geq (a_ma_{m+1}- a_{m-1}a_{m+2})^2.$$

\noindent Nicolas \cite{Nicolas78} and DeSalvo and Pak \cite{DP15} proved that the partition function $p(n)$ is log concave for $n \geq 25$, where $p(n)$ is the number of partitions of $n$. Recall that a partition of a positive integer $n$ with $r$ parts is a weakly decreasing sequence of  $r$ positive integers that sums to $n$.  We set $p(0) = 1$. Chen \cite{Chen17} conjectured that $p(n)$ satisfies the third order Tur\'{a}n  inequality for $n \geq 95$, which was proved by Chen, Jia and Wang \cite{CJW19}. Their result also shows that the cubic polynomial $$\sum_{k=0}^3\binom{3}{k} p(n + k)x^k$$ has only real simple roots for $n \geq 95$. They also conjectured that for $d \geq 4$, there exists a positive integer $N(d)$ such that the Jensen polynomials $J^{d,n}_p(X)$ for $p(n)$ as defined in (\ref{Jensen Poly}) have only real roots for all $n \geq N(d)$. Note that the case for $d = 1$ is trivial with $N(1) = 1$, and the log concavity of $p(n)$ for $n \geq 25$ proves the case for $d =2 $ with $N(2) = 25$.  Chen, Jia and Wang \cite{CJW19} proved the case for $d = 3$ with $N(3) = 94$. Larson and Wagner \cite{LW19} proved the minimum value of $N(d)$ for $d \leq 5$ and gave an upper bound for all other $N(d)$. The conjecture was proven for all $d \geq 1$ in a recent paper by Griffin, Ono, Rolen and Zagier \cite{GORZ_1} where they proved the hyperbolicity of the Jensen polynomials associated to a large family of sequences. Given an arbitrary sequence $\alpha = (\alpha(0), \alpha(1), \alpha(2), \cdots)$ of real numbers, the associated \textit{Jensen polynomial $J_{\alpha}^{d,n}(X)$ of degree $d$ and shift $n$} is defined by

\begin{equation} \label{Jensen Poly}
    J_{\alpha}^{d,n}(X) := \sum\limits_{j=0}^d \binom{d}{j} \alpha(n+j)X^j.
\end{equation}

\noindent The Jensen polynomials also have a close relationship to the Riemann Hypothesis. Indeed, P\'{o}lya \cite{Polya1927} proved that the Riemann Hypothesis is equivalent to the hyperbolicity of all of the Jensen polynomials $J^{d,n}_\gamma(X)$ associated to the Taylor coefficients $\{\gamma(j)\}_{j=0}^{\infty}$ of $\displaystyle\frac{1}{8}\Xi\bigg(\frac{i\sqrt{x}}{2}\bigg)$. Griffin, Ono, Rolen and Zagier \cite{GORZ_1} proved that for each $d \geq 1$, all but finitely many $J^{d,n}_\gamma(X)$ are hyperbolic, which provides new evidence supporting the Riemann Hypothesis.

There is a classical result by Hermite that generalizes the Tur\'{a}n  inequalities using Jensen polynomials. Let $$f(x) = x^n + a_{n-1}x^{n-1} + \cdots + a_1x + a_0$$ be a polynomial with real coefficients. Let $\beta_1, \beta_2, \cdots, \beta_n$ be the roots of $f$ and denote $S_0 = n$ and $$S_m = \beta_1^m + \beta_2^m + \cdots + \beta_n^m, \qquad m = 1, 2, 3, \cdots $$ their Newton sums. Let $M(f)$ be the Hankel matrix of $S_0, \cdots S_{2n-2}$, i.e. $$M(f) := \begin{pmatrix}
S_0 &S_1&S_2&\cdots&S_{n-1}\\
S_1 &S_2&S_3&\cdots&S_{n}\\
S_2 &S_3&S_4&\cdots&S_{n+1}\\
\vdots &\vdots& \vdots& \vdots&\vdots\\
S_{n-1} &S_n&S_{n+1}&\cdots&S_{2n-2}\\
\end{pmatrix}.$$

\noindent Hermite's theorem \cite{NO63} states that $f$ is hyperbolic if and only if $M(f)$ is positive definite. Recall that a polynomial with real coefficients is called hyperbolic if all of its roots are real. Each $S_m$ can be expressed in terms of the coefficients $a_0, \cdots, a_{n-1}$ of $f$ for $m \geq 1$, and a matrix is positive definite if and only if all its principle minors are positive. Thus Hermite's theorem provides a set of inequality conditions on the coefficients of a hyperbolic polynomial $f$: $$\Delta_1 = S_0 = n, \Delta_2 = \begin{vmatrix}S_0&S_1\\S_1&S_2\end{vmatrix} > 0, \cdots, \Delta_n = \begin{vmatrix}
S_0 &S_1&S_2&\cdots&S_{n-1}\\
S_1 &S_2&S_3&\cdots&S_{n}\\
S_2 &S_3&S_4&\cdots&S_{n+1}\\
\vdots &\vdots& \vdots& \vdots&\vdots\\
S_{n-1} &S_n&S_{n+1}&\cdots&S_{2n-2}\\
\end{vmatrix}> 0.$$

\noindent For a given sequence $\alpha(n)$, when Hermite's theorem is applied to $J_{\alpha}^{d,n}(X)$ then the condition that all minors $\Delta_k$ of the Hankel matrix $M(J_{\alpha}^{d,n}(X))$ are positive gives a set of inequalities on the sequence $\alpha(n)$. We call these inequalities, together with the equality cases, the \textit{order $k$ Tur\'{a}n inequalities}. In other words,  $J_{\alpha}^{d,n}(X)$ is hyperbolic  if and only if the subsequence $\{\alpha(n + j)\}_{j=0}^{\infty}$ satisfies all the order $k$ Tur\'{a}n (strict) inequalities for all $1\leq k \leq d$. In particular, the result in \cite{GORZ_1} shows that for any $d \geq 1$, the partition function $\{p(n)\}$ satisfies the order $d$ Tur\'{a}n inequality for sufficiently large $n$.


For a positive integer $k \geq 2$, let the \textit{k-regular partition function} $p_k(n)$ be defined as the number of partitions of $n$ in which none of the parts are multiples of $k$. For any fixed integer $k \geq 2$, the generating function for the sequence $\{p_k(n)\}_{n=0}^\infty$ is given by

$$\sum\limits_{n \geq 0} p_k(n) q^n = \prod\limits_{n=1}^\infty \dfrac{(1-q^{kn})}{(1-q^n)}.$$

\noindent Sloane conjectured that the sequence of partitions of $n$ with distinct parts, is log concave for sufficiently large $n$. It is well known that the number of $k$-regular partitions is equal to the number of partition with parts appearing at most $k-1$ times. Therefore, Sloane's conjecture is equivalent to the log-concavity of $p_2(n)$ for sufficiently large $n$. In this paper, we prove Sloane's conjecture, and in fact we prove the more general result that for any $d$, the $k$-regular partition functions $p_k(n)$ satisfy the order $d$ Tur\'{an} inequality for sufficiently large $n$.

As has been done for other sequences, we can define the Jensen polynomials $J_{p_k}^{d,n}(X)$ of degree $d$ and shift $n$ for the sequence $p_k(n)$ by setting $\alpha = p_k$ in (\ref{Jensen Poly}). In the spirit of much other work on the Tur\'{a}n inequalities, we may reframe the conjecture in terms of the hyperbolicity of the associated Jensen polynomials:

\begin{conjecture*}\label{NS conjecture}
	Let $d \geq 1, k \geq 2$ be integers. There exists a positive integer $N_k(d)$ such that the degree $d$ Jensen polynomial $J_{p_k}^{d,n}(X)$ associated to $p_k(n)$ are hyperbolic for all $n \geq N(d)$.
\end{conjecture*}

For a natural number $d$, we define the Hermite polynomials $H_d(X)$ by the generating function

\begin{equation*} 
    e^{-t^2 + Xt} = \sum\limits_{d=0}^\infty H_d(X) \dfrac{t^d}{d!} = 1 + Xt + (X^2 - 2)\dfrac{t^2}{2} + (X^3 - 6X) \dfrac{t^3}{6} + \cdots.
\end{equation*}

\noindent The Hermite polynomials are orthogonal polynomials with respect to a slightly different normalization than is standard. In view of these definitions, we can now state the main results of the paper.

\begin{theorem} \label{Main Theorem}
If $k \geq 2$ and $d \geq 1$, then $$\lim\limits_{n \to \infty} \widehat{J}^{d,n}_{p_k}(X) = H_d(X),$$ uniformly for $X$ on compact subsets of $\mathbb{R}$, where $\widehat{J}^{d,n}_{p_k}(X)$ are renormalized Jensen polynomials for $p_k(n)$ as defined in (\ref{J-Hat Def}).
\end{theorem}

\begin{corollary} \label{Corollary}
If $k \geq 2$ and $d \geq 1$, then $J^{d,n}_{p_k}(X)$ is hyperbolic for sufficiently large $n$.
\end{corollary}

\begin{remark}
By Corollary \ref{Corollary}, there exists a minimal natural number $N_k(d)$ such that $J^{d,n}_{p_k}(X)$ is hyperbolic for all $n \geq N_k(d)$. These numbers are not the focus of this paper, however a brief discussion is worthwhile, as these numbers dictate the effectiveness of the main theorem. The following table provides conjectural values of $N_k(d)$ for small $k$ and $d$.

\begin{center}
\begin{tabular}{|c|c|c|c|c|} \hline
$d$ & $N_2(d)$ & $N_3(d)$ & $N_4(d)$ & $N_5(d)$ \\ \hline
$2$ & $32$ & $57$ & $16$ & $41$ \\ \hline
$3$ & $120$ & $184$ & $63$ & $136$ \\ \hline
$4$ & $266$ & $390$ & $137$ & $294$ \\ \hline
\end{tabular}
\end{center}

\noindent We may also ask about trends as $k$ or $d$ vary. As $k$ increases, the fact that $p_k(n) = p(n)$ for all $n < k$ suggests that for any fixed $d$, $N_k(d) = N(d)$ for all sufficiently large $k$. The close connection between $p_k(n)$ and $p(n)$ also suggests that for $k$ fixed, $N_k(d)$ will grow like $N(d)$. In particular, it seems likely that $N_k(d)/N(d)$ is bounded (and possibly converges) as $d \rightarrow \infty$. In this way, any conjecture or theorem about $N(d)$ ought to have an analog for $N_k(d)$. These problems are interesting in their own right, but are not taken up here.
\end{remark}

\begin{remark}
We are motivated to study the Tur\'{an} inequalities for $p_k(n)$ because of Sloane’s conjecture. The work of Griffin, Ono, Rolen and Zagier \cite{GORZ_1} provides a criterion (see Theorem \ref{GORZ_1 Thm 3}) that can be used to establish the Tur\'{an} inequalities for Fourier coefficients of many modular forms. We employ this criterion in the proof of the Tur\'{an} inequalities for $p_k(n)$. One might suspect that the Tur\'{an} inequalities will hold for all modular forms which are eta-quotients. This is not true. First of all the inequalities require the Fourier coefficients are always nonnegative for large $n$. Moreover, the method of proof requires that a corresponding modular form have a pole at some cusps. The $k$-regular partition functions are perhaps the most natural family for which these requirements hold. Using our methods, similar results follow mutatis mutandis for other forms which satisfy these two conditions.
\end{remark}

In Section 2, we provide a proof of Theorem \ref{Main Theorem}, and thus of the conjecture of Sloane. We first present the P\'{o}lya-Jensen method for proving the hyperbolicity of Jensen polynomials. We then present the work of Hagis \cite{Hagis71} on an explicit formula for $p_k(n)$. Finally, we prove Theorem \ref{Main Theorem} by using asymptotic expansions of Hagis' formulas to derive the necessary conditions for the use of the P\'{o}lya-Jensen method.

\section{Proof of Theorem \ref{Main Theorem}}
\subsection{Polya-Jensen method}

In 1927, P\'{o}lya \cite{Polya1927} demonstrated that the Riemann hypothesis is equivalent to the hyperbolicity of the Jensen polynomials for $\zeta(s)$ at its point of symmetry. This approach to the Riemann hypothesis has not seen much progress until the recent work of Griffin, Ono, Rolen and Zagier \cite{GORZ_1}, in which the hyperbolicity of these polynomials is proved for all $d \geq 1$ for sufficiently large $n$. The primary method utilized by this paper can be stated as follows:

\begin{theorem}[Theorem 3 \& Corollary 4, \cite{GORZ_1}]\label{GORZ_1 Thm 3}
Let $\{ \alpha(n) \}$, $\{ A(n) \}$, and $\{ \delta(n) \}$ be sequences of positive real numbers such that $\delta(n) \to 0$ as $n \to \infty$. Suppose further that for a fixed $d \geq 1$ and for all $0 \leq j \leq d$, we have

\begin{equation*}
    \log\bigg( \dfrac{\alpha(n+j)}{\alpha(n)} \bigg) = A(n)j - \delta(n)^2j^2 + \sum_{i = 3}^d g_i(n) j^i + o\big( \delta(n)^d \big) \hspace{.3in} \text{ as } n \to \infty.
\end{equation*}

\noindent where $g_i(n) = o(\delta(n)^i)$ for each $0 \leq i \leq d$. Then the renormalized Jensen polynomials  $\widehat{J}^{d,n}_{\alpha}(X) = \dfrac{\delta(n)^{-d}}{\alpha(n)} J^{d,n}_{\alpha}\bigg( \dfrac{\delta(n)X - 1}{\exp(A(n))}\bigg)$ satisfy $\lim\limits_{n \to \infty} \widehat{J}^{d,n}_{\alpha}(X) = H_d(X)$ uniformly for $X$ in any compact subset of $\mathbb{R}$. Furthermore, this implies that the polynomials $J^{d,n}_{\alpha}(X)$ are hyperbolic for all but finitely many values of $n$.
\end{theorem}

\begin{remark}
The statement of this theorem in \cite{GORZ_1} has a typo. The correct statement of the logarithm condition is stated in (15) of \cite{GORZ_1}.
\end{remark}

Because the conditions for this result are so general, the method can be utilized in a wide variety of circumstances. For instance, it is shown in Theorem 7 of \cite{GORZ_1} that if $a_f(n)$ are the (real) Fourier coefficients of a modular form $f$ on $\text{SL}_2(\mathbb{Z})$ holomorphic apart from a pole at infinity, then there are sequences $A_f(n)$ and $\delta_f(n)$ such that $\alpha(n) = a_f(n)$ satisfies the required conditions. What we prove can then be regarded as a higher-level generalization of this result, since the sequences $p_k(n)$ are coefficients of weight zero weakly holomorphic modular forms on proper subgroups of $SL_2(\mathbb{Z})$. 

\subsection{Asymptotic Behavior of $p_k(n)$}

Hagis' \cite{Hagis71} results on $p_k(n)$ play a key role in the main theorem. Hagis proved an explicit formula for $p_k(n)$ very similar to Rademacher's famous explicit formula for $p(n)$. Hagis uses the Hardy-Ramanujan circle method to obtain his formula for $p_k(n)$, and the same formula can be found by expressing a generating function for $p_k(n)$ by quotients of Dedekind's eta function and expanding the resulting expression by Poincare series. The resulting formula expressible as a sum of modified Kloosterman sums times Bessel functions. Using facts about the asymptotics of Bessel functions and the explicit formulas, useful asymptotics for $p_k(n)$ may be derived. In particular, Hagis proves as a corollary (Corollary 4.1 in \cite{Hagis71}) the asymptotic formula

\begin{equation} \label{Hagis Asymptotic}
    p_k(n) = 2\pi \sqrt{\dfrac{m_k}{k(n + k m_k)}} \cdot I_1\bigg(4\pi \sqrt{m_k(n + k m_k)} \bigg) (1 + O(\exp(-cn^{1/2}))),
\end{equation}

\noindent where $I_1$ is a modified Bessel function of the first kind, $c = c(k)$ is a constant, and $m_k = (k-1)/24k$.

\subsection{Proof of Theorem \ref{Main Theorem}}

We now begin the proof of Theorem \ref{Main Theorem}.

\begin{proof}[Proof of \ref{Main Theorem}]
Fix $d \geq 1$ and $k \geq 2$, and let the sequences $A_k(n)$, $\delta_k(n)$ be defined by

\begin{equation*}
    A_k(n) = 2\pi \sqrt{m_k/n} + \dfrac{3}{4}\sum\limits_{r=1}^{\lfloor 3d/4 \rfloor} \dfrac{(-1)^r}{rn^r} \ \ \text{and} \ \  \delta_k(n) = \bigg(- \sum\limits_{r=2}^\infty  \dfrac{4\pi \sqrt{m_k} \binom{1/2}{r}}{n^{r-1/2}} \bigg)^{1/2}.
\end{equation*}

\noindent Define the renormalized Jensen polynomials $\widehat{J}^{d,n}_{p_k}(X)$ by

\begin{equation} \label{J-Hat Def}
    \widehat{J}^{d,n}_{p_k}(X) := \dfrac{\delta_k(n)^{-d}}{p_k(n)} J^{d,n}_{p_k}\bigg( \dfrac{\delta_k(n)X - 1}{\exp(A_k(n))} \bigg).
\end{equation}

\noindent By application of the Jensen-P\'{o}lya method, it suffices to show that for any fixed $d$ and all $0 \leq j \leq d$,

\begin{equation} \label{Log-Quotient Eqn}
\log\bigg( \dfrac{p_{k}(n+j)}{p_k(n)} \bigg) = A_k(n)j - \delta_k(n)^2j^2 + \sum_{i=3}^k g_i(n) j^i + o\big( \delta_k(n)^d \big) \hspace{.3in} \text{ as } n \to \infty.
\end{equation}
\vspace{.1in}

\noindent where $g_i(n) = o(\delta_k(n)^i)$ holds for each $i$. Using (\ref{Hagis Asymptotic}), we have
\begin{equation*}
    p_k(n) = b_k (n + k m_k)^{-1/2} I_1(4\pi \sqrt{nm_k}) + O(n^{d_k} e^{-c_k \sqrt{n}}),
\end{equation*}
\noindent as $n \to \infty$, where $b_k, c_k > 0$, and $d_k$ are constants which depend at most on $k$. In light of the expansion of the Bessel functions of the first kind at infinity, this implies that $p_k(n)$ has asymptotic expansion to all orders of $1/n$ in the form

$$p_k(n) \sim e^{4\pi \sqrt{nm_k}} n^{-3/4} \exp\bigg( \sum_{r=0}^\infty \frac{a_r}{n^r}\bigg),$$

\noindent where $a_0, a_1, \cdots$ are constants depending only on $k$. Furthermore, when the exponential terms are considered asymptotically, the terms $\dfrac{a_r}{n^r}$ in the sum vanish compared with the term $4\pi \sqrt{nm_k}$ for $r \geq 1$, and so we have

$$p_k(n) \sim e^{a_0 + 4\pi \sqrt{nm_k}} n^{-3/4}.$$

\noindent It follows that for fixed $0 \leq j \leq d$, we have

\begin{eqnarray*}
&&\log\bigg( \dfrac{p_k(n+j)}{p_k(n)} \bigg)\\&\sim& 4\pi \sqrt{m_k} \sum\limits_{r=1}^\infty \binom{1/2}{r} \dfrac{j^r}{n^{r-1/2}} - \dfrac{3}{4} \sum\limits_{r=1}^\infty \dfrac{(-1)^{r-1}j^r}{rn^r}\\
&=& 2\pi \sqrt{m_k}\dfrac{j}{\sqrt{n}}+ 4\pi \sqrt{m_k}\sum\limits_{r=2}^\infty \binom{1/2}{r} \dfrac{j^r}{n^{r-1/2}} + \dfrac{3}{4} \sum\limits_{r=1}^\infty \dfrac{(-1)^{r}j^r}{rn^r}\\
&=& 2\pi \sqrt{\dfrac{m_k}{n}}j  + \dfrac{3}{4} \sum\limits_{r=1}^{\lfloor 3d/4\rfloor} \dfrac{(-1)^{r}j^r}{rn^r}+4\pi \sqrt{m_k}\sum\limits_{r=2}^\infty \binom{1/2}{r} \dfrac{j^r}{n^{r-1/2}}+ \dfrac{3}{4} \sum\limits_{r=\lfloor 3d/4\rfloor + 1}^\infty \dfrac{(-1)^{r}j^r}{rn^r}\\
&=& \Bigg(2\pi \sqrt{\dfrac{m_k}{n}}+ \dfrac{3}{4} \sum\limits_{r=1}^{\lfloor 3d/4\rfloor} \dfrac{(-1)^{r}j^{r-1}}{rn^r}\Bigg)j+\Bigg(\sum\limits_{r=2}^\infty  \dfrac{4\pi \sqrt{m_k}\binom{1/2}{r}j^{r-2}}{n^{r-1/2}}\Bigg)j^2 \\ && \hspace{7.3cm} + \hspace{0.1cm} \dfrac{3}{4} \sum\limits_{r=\lfloor 3d/4 \rfloor + 1}^\infty \dfrac{(-1)^{r}j^r}{rn^r}.\\
\end{eqnarray*}

 \noindent Then (\ref{Log-Quotient Eqn}) follows since $\log\bigg( \dfrac{p_{k}(n+j)}{p_k(n)} \bigg) - A_k(n)j + \delta_k(n)^2 j^2 = O(n^{- \lfloor 3d/4 \rfloor - 1}) = o(\delta_k(n)^d)$.
\end{proof}

\section*{Acknowledgements}

We would like to thank Ken Ono for suggesting this problem and his guidance. We would also like to thank Cormac O'Sullivan and Jacques Gélinas for comments on an earlier version of the paper.

\bibliographystyle{plain}
\bibliography{citation}

\end{document}